\documentclass[11pt, a4paper]{article}

\usepackage[paper=a4paper,top=2cm,bottom=3cm]{geometry}
\usepackage[UKenglish]{babel}
\usepackage{amssymb, amsthm, amsmath}
\usepackage{enumerate}
\usepackage{mathrsfs}
\usepackage[babel, ,german=quotes]{csquotes}
\usepackage{nccfoots}

\newtheorem{thm}{Theorem}[section]

\newtheorem{prop}[thm]{Proposition}
\newtheorem{lem}[thm]{Lemma}

\newtheorem{ass}[thm]{Assumption}
\theoremstyle{definition}
\newtheorem*{rem}{Remark}

\newenvironment{abs}{\begin{trivlist}\item[] {\bfseries Abstract.}}{\end{trivlist}}
\newenvironment{keywords}{\begin{trivlist}\item[] {\bfseries Keywords.}}{\end{trivlist}}
\newenvironment{AMS}{\begin{trivlist}\item[] {\bfseries {AMS subject classifications}.}}{\end{trivlist}}

\def\bmat{\begin{bmatrix}}
\def\emat{\end{bmatrix}}
\newcommand{\beq}[1]{\begin{equation}\label{#1}}
\def\eeq{\end{equation}}

\newcommand{\bit}{\begin{itemize}}
\newcommand{\eit}{\end{itemize}}
\newcommand{\bal}{\begin{align}}
\newcommand{\eal}{\end{align}}
\newcommand{\bma}{\left[\begin{array}{*{10}{r}}}
\newcommand{\ema}{\end{array}\right]}
\newcommand{\bmac}{\left[\begin{array}{*{10}{c}}}

\newcommand{\eps}{\varepsilon}
\newcommand{\vet}{\tilde{v}^{\eps}_t}
\newcommand{\yet}{y_t^{\eps}}
\newcommand{\intinf}{\int_{-\infty}^{\infty}}

\newcommand{\id}{\operatorname{Id}}

\DeclareFontFamily{U}{msb}{}
\DeclareFontShape{U}{msb}{m}{n}{ <5> <6> <7> <8> <9> gen * msbm
       <10> <10.95> <12> <14.4> <17.28> <20.74> <42.88> msbm10}{}
\DeclareSymbolFont{AMSb}{U}{msb}{m}{n}
\DeclareSymbolFontAlphabet{\mathbb}{AMSb}

\newcommand{\R}{\mathbb{R}}

\title{\bf{Front Propagation in Stochastic Neural Fields: A Rigorous Mathematical Framework}\Footnotetext{}{Date: \today}}

\author{\small{\uppercase{Jennifer Kr\"uger}}\footnote{Mathematisches Institut, Technische Universit\"at Berlin, 10587 Berlin, Germany (jkrueger@math.tu-berlin.de)}, \;\; \small{\uppercase{Wilhelm Stannat}}\footnote{Mathematisches Institut and Bernstein Center for Computational Neuroscience, Technische Universit\"at Berlin, 10587 Berlin, Germany (stannat@math.tu-berlin.de). The work of this author was supported by the BMBF, FKZ01GQ1001B. }}

\date{}

\begin{document}
\addtocounter{footnote}{1}
\maketitle

\begin{abs}
We develop a complete and rigorous mathematical framework for the analysis of stochastic neural field equations under the influence of spatially extended additive noise. By comparing a solution to a fixed deterministic front profile it is possible to realise the difference as strong solution to an $L^2(\R)$-valued SDE. A multiscale analysis of this process then allows us to obtain rigorous stability results. Here a new representation formula for stochastic convolutions in the semigroup approach to linear function-valued SDE with adapted random drift is applied.  Additionally, we introduce a dynamic phase-adaption process of gradient type. 
\end{abs}


\section{Introduction}
Neural field equations are used to model the spatiotemporal evolution of  neural activity in thin layers of cortical tissue from a macroscopic point of view. 
Under the influence of spatially extended additive noise $\xi$ they can be formulated as
\begin{align}\label{nfe}
\partial_t u(x,t)=-u(x,t)+\intinf w(x-y)\ F(u(y,t)) \ dy + \sqrt{\eps}\ \dot{\xi}(x,t)
\end{align}
where $u$ is meant to describe the activity of a neuron at position $x$ and time $t$. The probability kernel $w$ models the strength of nonlocal excitatory synaptic connections, whereas $F: \R \to \R$ constitutes a nonlinear firing rate function. A precise interpretation of the stochastic forcing term $\xi$, characterising extrinsic fluctuations of the system, will be given below.
Several approaches (e.g. in \cite{Baladron}, \cite{Jirsa} and \cite{Touboul}) have been developed to derive neural field equations as continuum limit of spatially extended, synaptically coupled neural systems under different assumptions on the underlying microscopic network structure. 
Here, the particular importance of a macroscopic perspective on neural activity primarily arises from the capability of field equations like \eqref{nfe} to model a wide range of neurophysiological phenomena.
In the deterministic theory ($\eps=0$) Amari first provided a complete taxonomy of pattern dynamics exhibited by such fields (cf. \cite{Amari}), which includes the propagation of activity in form of travelling waves. For the simplified case that each neuron only switches between a spiking and a non-spiking state, i.e. that the firing rate function $F$ is given by a Heaviside nonlinearity, even an explicit travelling wave solution of \eqref{nfe} can be constructed (cf. \cite{Bressloff}), whereas for more general input data $w$ and $F$ sufficient criteria for the existence and uniqueness of such solutions have been formulated in \cite{Ermentrout} resp. \cite{Chen}. Extending the deterministic field equation to scalar stochastic neural fields, the objective of Bressloff and Webber in \cite{Bressloff} was to investigate the effect of spatially extended, extrinsic stochastic perturbations on such travelling wave dynamics. As it turned out, random forcing terms result in two distinct structural phenomena: `fast' perturbations of the front shape as well as a `slow' horizontal displacement of the wave profile from its uniformly translating position. 
Under these observations a separation of time-scales method was applied to solve a neural field equation (subject to multiplicative noise) by decomposing the solution into a fixed deterministic front profile, a diffusion-like horizontal translation process as well as time- and space-dependent fluctuations. However, as the authors of \cite{Bressloff} themselves state, they ``base [their] analysis on formal perturbation methods developed previously for PDEs, since a rigorous mathematical framework is currently lacking''.
\smallskip

The purpose of our paper therefore is to introduce such a complete and mathematically rigorous framework, which qualitatively captures the above phenomena and - by comparing a solution of \eqref{nfe} to a fixed reference profile - 
allows us to realise a stochastic neural field as stochastic evolution equation on a suitable function space (for the underlying theory refer to \cite{DaPrato}, \cite{Prevot}). Here, $\xi$ will be chosen as $Q$-Wiener process on the Hilbert space $L^2(\R)$. In contrast to the more direct rigorous approach recently suggested in \cite{Faugeras}, where a stochastic neural field is interpreted as SDE on a suitable weighted $L^2$-space, our approach via decomposition
has the advantage of providing information on the precise structure of solutions and even more, allows us to prove new stability results.

\smallskip 

The paper is structured as follows: Section \ref{sec:2} provides a brief introduction to the representation of stochastic travelling waves as developed in \cite{Bressloff}. These ideas are then given rigorous meaning by the decompositions presented in Section \ref{sec:3}. As part of this framework, in Subsection \ref{sec:4} we introduce a gradient-descent-type ODE for a dynamic phase-adaption and prove existence and uniqueness of a classical solution to this ODE in Theorem \ref{prop:ode}. Moreover, the fluctuations are further specified by a suitable decomposition into processes of different order, namely into an Ornstein-Uhlenbeck-like process satisfying a linear SDE with non-autonomous adopted random drift as well as a lower-order differentiable remainder process on $L^2(\R)$. As a consequence of the randomness of the linear drift term it is not possible to directly define a stochastic convolution of the (possibly non-adapted) semigroup with the driving Wiener process as required for a mild-solution concept. We therefore represent the Ornstein-Uhlenbeck process via a stochastic convolution of new type (recently introduced in \cite{Pronk}) that is applicable to general adapted random drifts and allows us to derive a locally uniform (in time) pathwise control on the Ornstein-Uhlenbeck component. A similar approach based on a pathwise control of the stochastic convolution has also been taken in \cite{StannatElSarhir}. Section \ref{sec:stab} then contains our main results, Theorem \ref{prop:y} and Theorem \ref{thm:stab}, on the stability of stochastic travelling waves. The crucial assumption to obtain this type of stochastic stability is the existence of a spectral gap of the time-dependent random drift. 
\smallskip 

An interesting issue for future research would be to investigate the underlying statistics of the dynamic wave speed as well as the Ornstein-Uhlenbeck component of the fluctuations, which, unlike in the case of deterministic associated drift operators, cannot be expected to be a Gaussian process. Moreover, our hope is that the presented approach can be carried over to spatially discrete models in order to examine the stochastic stability of fluctuating travelling waves in discrete neural networks.

\section{Phenomenological Motivation}\label{sec:2}
Interpreting \eqref{nfe} as stochastic evolution equation on a suitable function space we now analyse the following nonlinear, scalar neural field equation
\begin{align}\label{eq:u}
\begin{cases} 
du_t&= [-u_t+w\ast F(u_t)]\ dt + \sqrt{\eps} \ dW_t^Q,\\
u(0)&=u_0
\end{cases}
\end{align}

\noindent where the additive stochastic forcing term is given by a  $Q$-Wiener process $W^Q$ on $L^2(\R)$ (for a detailed introduction of $Q$-Wiener processes on general Hilbert spaces refer to \cite{DaPrato}). Let the weight kernel $w$ be a bounded probability density function, i.e. in particular
\[
\int_{\R} w(x)\ dx = 1
\]
and let the nonlinear gain function $F$ be in $C_b^2(\R)$, i.e. twice continuously differentiable with bounded derivatives.

Even though \cite{Bressloff} examined the effect of extrinsic multiplicative noise on the above field equation, the structural separation of the dynamics on different time-scales is also adoptable to the setting of additive extrinsic forcing terms as assumed in \eqref{eq:u}. Basically, the analysis in \cite{Bressloff} splits up the dynamics into its behaviour on short and long time-scales, i.e. the process is 
represented via a ``slow'' diffusive-like displacement of the front from its uniformly translating position and ``fast'' fluctuations in the front profile. More precisely, a fixed deterministic wave profile $U_0$ with uniformly translating position $\xi=x-ct$ is horizontally displaced by $\Delta(t)$, a diffusive-like stochastic process on the real line. Already taking account of the correct order w.r.t. $\eps$ the ansatz in \cite{Bressloff} is to solve \eqref{eq:u} by decomposing
\begin{align}\label{2.18}
u(x,t)=U_0(\xi-\Delta(t))+\sqrt{\eps}\ \Phi(\xi-\Delta(t),t),
\end{align}

\noindent where $U_0$ is a travelling wave solution of \eqref{eq:u} in the deterministic case $\eps=0$
and $\Phi$ can \mbox{- at least formally -} be derived as a process on $L^2(\R)$. For more detailed insights into the nature of these processes as well as the applied methodology we refer to the original paper. In the following sections this phenomenological motivation is translated into a rigorous mathematical analysis, which will partly proceed along lines of a similar approach having been developed in the context of stochastic reaction diffusion equations (cf. \cite{Stannat}, \cite{Veraar}).
\section{Mathematical Modelling on $L^2(\R)$}\label{sec:3}

In the sequel assume the input data $w$ and $F$ to allow for a unique travelling wave solution $\hat{u}$ with intrinsic wave speed $c$ solving the deterministic neural field equation \eqref{eq:u} in the case $\eps=0$. In addition, the wave profile should connect two stable fixed points 0 and 1 of the dynamics, i.e. we impose the boundary conditions $\lim_{\xi\to\infty}\hat{u}(\xi)=1, \; \lim_{\xi\to-\infty}\hat{u}(\xi)=0$. Sufficient criteria on $w$ and $F$ ensuring the existence and uniqueness of such solutions are stated in \cite{Ermentrout} resp. \cite{Chen}.

\subsection{Decomposition of the solution}
In analogy with \cite{Bressloff} we decompose the solution of the stochastic neural field \eqref{eq:u} into 
\begin{align}\label{eq:dec}
u_t= \hat{u}_t+v_t,
\end{align}

\noindent with $\hat{u}_t=\hat{u}(\cdot-ct)$ being a travelling wave solution of the deterministic ODE 
\begin{align}\label{eq:tw}
-c\hat{u}_x&= -\hat{u}+w\ast F(\hat{u})
\end{align}

\noindent Since by classical calculus $\hat{u}_t$ is a solution of
\begin{align}\notag
d\hat{u}_t(x) &= \partial_t \hat{u}(x-ct)\ dt= -c \hat{u}_x(x-ct) \ dt\\
&= \left[-\hat{u}_t(x)+ \left(w \ast F(\hat{u}_t)\right)(x)\right] dt
\end{align}
$v_t$ satisfies the following evolution equation
\begin{align}\notag
dv_t&=du_t-d\hat{u}_t\\ 
&=[-v_t+w\ast \left(F(v_t+\hat{u}_t)-F(\hat{u}_t)\right)]\ dt +\sqrt{\eps} dW_t^Q \label{eq:v1}
\end{align}

\noindent Expanding to first order yields
\begin{align}\label{eq:v2}
dv_t= [-v_t+w\ast \left(F'(\hat{u}_t)\ v_t\right)]\ dt + R_t(\hat{u}_t,v_t) \ dt +\sqrt{\eps} dW_t^Q,
\end{align}

\noindent where the remainder is given by
\begin{align}\label{remainder}
R_t(\hat{u}_t, v_t)&= w\ast\left(\frac 12 F''(\xi(\hat{u}_t,v_t))\  v_t^2\right)
\end{align}

\noindent with $\xi(\hat{u}_t,v_t)$ denoting an intermediate point between $\hat{u}_t$ and $v_t$.
For a given travelling wave solution $\hat{u}$ the remainder $R_t(\hat{u}_t, \cdot)$ is indeed a well-defined map on $L^2(\R)$ satisfying the estimate

\begin{align}\label{est:rem} 
\Vert R_t(\hat{u}_t,v_t)\Vert_{L^2}^2
&\leq c \int_{\R}\left( \int_{\R}  w(x-y)\ v_t^2(y) \ dy\right)^2  dx= \tilde{c} \ \Vert v_t\Vert_{L^2}^4.
\end{align}

\begin{thm}\label{thm:v}
For $v_0\in L^2(\R)$ equation \eqref{eq:v1} has a unique mild solution $v$ with $v \in L^{\infty}([0,T]; L^2(\R))$ almost surely. This solution is also a strong solution and admits a continuous modification.
\end{thm}

\begin{proof}
For proving existence of a unique mild solution we define
\[
B(t,v):= -v+w \ast \left( F(v+\hat{u}_t)-F(\hat{u}_t)\right).
\]
Applying Jensen's inequality we obtain that $B$ is well-defined as a map $[0,T]\times L^2(\R)\to L^2(\R)$ satisfying the following estimate
\begin{align*}
\Vert B(t,v)\Vert_{L^2(\R)}^2
&\leq  2\Vert v\Vert_{L^2(\R)}^2+ 2\Vert F'\Vert_{\infty}^2\ \Vert v\Vert^2_{L^2(\R)}. 
\end{align*}

\noindent For a given topological space $X$ let $\mathcal{B}(X)$ denote the Borelian $\sigma$-algebra on $X$. Then, $B$ is obviously measurable from the measurable space
${([0,T]\times L^2(\R),\ \mathcal{B}([0,T])\otimes \mathcal{B}(L^2(\R)))}$ to $(L^2(\R),\ \mathcal{B}(L^2(\R)))$.
\medskip

\noindent As one can easily verify for $v_1, v_2\in L^2(\R), \ t\in[0,T]$, we obtain the Lipschitz property
\begin{align*}
\Vert B(t,v_1)-B(t,v_2)\Vert_{L^2(\R)}^2
\leq 2\ (1+  \Vert F'\Vert_{\infty}^2)\ \Vert v_1-v_2\Vert_{L^2(\R)}^2,
\end{align*}
which, together with the property 
\[
B(t,0)=w\ast \left(F(0+\hat{u}_t)-F(\hat{u}_t)\right) = 0,
\]
directly implies a linear growth condition for $B$. 
Considering \eqref{eq:v1} as semilinear evolution equation with linear part $A v_t$, $A=0$, \cite[Theorem 7.4]{DaPrato} provides the existence of a unique mild solution
$v\in L^{\infty}([0,T], L^2(\R))$ a.s., which allows for a continuous modification and can be represented as

\begin{align}\label{mil}
v_t= v_0+\int_0^t B(s, v_s) \ ds + \sqrt{\eps} W_t^Q.
\end{align}
This immediately shows that $v$ is also a strong solution of \eqref{eq:v1}.
\end{proof}

It is clear from the above that if the initial condition $u_0$ of problem \eqref{eq:u} is chosen such that $v_0:= u_0-\hat{u}_0 \in L^2(\R)$, then $u_t:=\hat{u}_t+v_t$ is a solution of \eqref{eq:u}. 
Uniqueness of the components only holds after we have fixed an initial phase for the travelling wave solution $\hat{u}_t$.

\subsection{Gradient-descent type ODE for random phase-shift}\label{sec:4}
The intention of this subsection is to find a suitable derivation of a slow horizontal translation process $(C(t))_{t\in[0,T]}$, as mentioned in Section \ref{sec:2}, such that the initial representation $u_t=\hat{u}_t + v_t$ can be replaced by 

\begin{align}
u_t=\hat{u}(\cdot-ct-C(t))+\tilde{v}_t,
\end{align}
where $\tilde{v}_t$ denotes a new $L^2$-component. The phase-shift process $(C(t))$ should have the effect of dynamically matching the deterministic profile $\hat{u}$ with the stochastic travelling wave in the sense of minimising the $L^2$-distance between the solution $u$ and all possible translations of $\hat{u}$, for which reason (as already proposed in \cite{Thuemmler},  \cite{Stannat}) we consider the gradient-descent-type (pathwise) ODE
 
\begin{align}\label{gradientdesc}
\begin{cases}
\dot{C}(t) = -m\ \langle \hat{u}_x(\cdot-ct-C(t)), u(t,\cdot)-\hat{u}(\cdot -ct-C(t))\rangle_{L^2}\\
C(0)=0
\end{cases}
\end{align}
with relaxation rate $m>0$. This phase-adaption can be seen as an alternative approach to the phase conditions specified by certain algebraic constraints in the classical stability analysis (refer to \cite{Henry}).
Existence and uniqueness of a classical global solution to the ODE \eqref{gradientdesc} are proven under the 
following assumptions on the weight kernel:

\begin{ass}\label{ass:w}
\begin{enumerate}[(a)]
\item $w$ is piecewise continuously differentiable 
\item $w$ satisfies
\[
\int_{\R} \frac{w_x^2}{w}(x)\ dx < \infty \hspace{10cm}\hphantom{4}
\]
\end{enumerate}

\end{ass}

\vspace{0.3cm}

\begin{rem}\label{ass:lipgrad}

\begin{enumerate}[(i)]
\item 
Note that Assumption \ref{ass:w}(b) ensures that in the case $c=0$ the travelling wave satisfies $\hat{u}_{xx}\in L^2(\R)$, which is used in the proof of Lemma \ref{lem:lipgrad}(ii). In the case $c\neq 0$ part (b) of the above assumption is not needed.
\item The above conditions are satisfied for the exponential weight function ${w(x) = \frac{1}{2\sigma}e^{-\frac{\vert x \vert}{\sigma}}}$ as well as for a Gaussian kernel $w(x)= \frac{1}{\sqrt{2\pi\sigma^2}}\ e^{-\frac{x^2}{2\sigma^2}}$, which are both common choices in modelling synaptic excitatory connections.
\end{enumerate}
\end{rem}

\begin{lem}
For bounded weight kernel $w$ and $F\in C_b^1(\R)$ we obtain
\[\hat{u}_x \in L^2(\R).\]
\end{lem}

\begin{proof}
In the case $c\neq 0$ multiplying \eqref{eq:tw} with the gradient $\hat{u}_x$ yields
\[
\hat{u}_x^2= \frac 1c \hat{u}\ \hat{u}_x-\frac 1 c\left( w\ast F(\hat{u})\right)\hat{u}_x
\]
such that by integration by parts for $y,z\in\R$
\begin{align*}
\int_y^z \hat{u}_x^2(r) dr &\leq \frac 1{2 \vert c\vert}\left(\hat{u}^2(z)+\hat{u}^2(y)\right) + \frac 1 {\vert c\vert} \Vert w\ast F(\hat{u})\Vert_{\infty} (\hat{u}(z)+ \hat{u}(y))\\
&\longrightarrow \frac 1{2 \vert c\vert}+ \frac 1 {\vert c\vert} \Vert w\ast F(\hat{u})\Vert_{\infty} \;\;\text{ as }y\to-\infty, \; z\to\infty.
\end{align*}
Here, the boundary conditions of $\hat{u}$ were applied.
\medskip

\noindent In the case $c=0$ \eqref{eq:tw} yields $\hat{u}=w\ast F(\hat{u})$
and under Assumption \ref{ass:w} (b) one obtains
\[
\hat{u}_x= w\ast F'(\hat{u})\hat{u}_x.
\]

\noindent Note that trivially for all $c\in\R$ we have $\hat{u}_x \in L^1(\R)$. This suffices to obtain
\begin{align*}
\Vert \hat{u}_x \Vert^2_{L^2(\R)} &= \int \Big{\vert} \int w(x-y)F'(\hat{u}(y))\hat{u}_x(y) \ dy \Big{\vert}^2 dx\\
&= \int\int\int w(x-y) w(x-\tilde{y}) F'(\hat{u}(y)) F'(\hat{u}(\tilde{y})) \hat{u}_x(y) \hat{u}_x(\tilde{y}) \ dy d\tilde{y} dx\\
&\leq \left(\sup_{y, \tilde{y}} \int w(x-y)w(x-\tilde{y})\ dx\right)
\Vert F'\Vert_{\infty}^2 \Vert \hat{u}_x\Vert_{L^1(\R)}^2\\
&\leq \Vert w\Vert_{\infty} \Vert F'\Vert_{\infty}^2 \Vert \hat{u}_x\Vert_{L^1(\R)}^2
\end{align*}
\end{proof}

\noindent With the above considerations the following Lipschitz properties for $\hat{u}$ as well as $\hat{u}_x$ are obtained:
\begin{lem}\label{lem:lipgrad}
Given Assumption \ref{ass:lipgrad} let $C_1, C_2\in\R$. Then 
\begin{enumerate}[(i)]
\item 
\[
\Vert \hat{u}(\cdot-C_1)-\hat{u}(\cdot-C_2)\Vert_{L^2} \leq \Vert \hat{u}_x \Vert_{L^2}\ \vert C_1-C_2\vert 
\]
\item There exists a constant $\tilde{c}>0$ such that
\[
\Vert \hat{u}_x(\cdot-C_1)-\hat{u}_x(\cdot-C_2)\Vert_{L^2} \leq \tilde{c}\ \vert C_1-C_2\vert
\]
\end{enumerate}
\end{lem}

\begin{proof} 
The proof of property (i) only requires H\"older's inequality, whereas for (ii)
note that in the case $c\neq 0$ equation \eqref{eq:tw} determines the gradient of the travelling wave by
\[
\hat{u}_x=\frac{1}{c} (\hat{u}-w\ast F(\hat{u})).
\]
Then, (i) and Jensen's inequality suffice to verify the above Lipschitz property. In the case $c=0$ equation \eqref{eq:tw} yields the identity $\hat{u}=w\ast F(\hat{u})$ such that by differentiating we see that the gradient satisfies
$\hat{u}_x=w\ast F'(\hat{u}) \hat{u}_x$. The rest again follows from straightforward calculations and Assumption \ref{ass:w}.
\end{proof}

After these preparatory considerations we state the main existence and uniqueness result of this subsection:
\begin{prop} \label{prop:ode}
Let $m>0$ and let Assumption \ref{ass:lipgrad} be satisfied. Then
$P$-almost surely there exists a unique global solution of the (pathwise) ODE \eqref{gradientdesc}.
\end{prop}

\begin{proof}
Let 
\[
B(t,C):=\langle \hat{u}_x(\cdot-ct-C), u(t,\cdot)-\hat{u}(\cdot -ct-C)\rangle_{L^2}.
\]
Since \eqref{gradientdesc} constitutes an initial value problem of first order it suffices to show that the map 
${(t,C)\in[0,T]\times \R \mapsto B(t,C)}$ is continuous and Lipschitz continuous w.r.t. $C$ uniformly in $t$. Decomposing
\begin{align}\notag
B(t,C) &= \langle \hat{u}_x(\cdot-ct-C), u(t,\cdot)-\hat{u}(\cdot -ct)\rangle_{L^2} + \langle \hat{u}_x(\cdot-ct-C), \hat{u}(\cdot-ct)-\hat{u}(\cdot -ct-C)\rangle_{L^2}
\end{align}

\noindent observe that the maps
\begin{enumerate}[(i)]
\item $(t,C)\mapsto \hat{u}_x(\cdot -ct-C)$
\item $(t,C)\mapsto u(t,\cdot)-\hat{u}(\cdot-ct)$
\item $(t,C)\mapsto \hat{u}(\cdot-ct)-\hat{u}(\cdot -ct-C)$
\end{enumerate}
are continuous in $L^2(\R)$. By Lemma \ref{lem:lipgrad} this holds true due to the continuity of $\hat{u},\ \hat{u}_x$ and ${v_t=u(t,\cdot)-\hat{u}_t}$, where we take the continuous modification provided by Theorem \ref{thm:v}. Next, for any $C_1, C_2 \in \R, \;t\in[0,T]$,

\begin{align*}
B(t,C_1)-B(t,C_2)&= \langle \hat{u}_x(\cdot-ct-C_1)-\hat{u}_x(\cdot-ct-C_2) , u(t,\cdot)-\hat{u}(\cdot -ct)\rangle_{L^2}\\
&\quad+ \langle \hat{u}_x(\cdot-ct-C_1), \hat{u}(\cdot-ct)-\hat{u}(\cdot -ct-C_1)\rangle_{L^2}\\
&\quad-\langle \hat{u}_x(\cdot-ct-C_2), \hat{u}(\cdot-ct)-\hat{u}(\cdot -ct-C_2)\rangle_{L^2}\\
&= I+ II+ III, \text{ say.}
\end{align*}

\noindent Applying Cauchy-Schwarz as well as Lemma \ref{lem:lipgrad} the first summand is $P$-a.s. controlled by
\begin{align*}
\vert I \vert &\leq \Vert \hat{u}_x(\cdot-ct-C_1)-\hat{u}_x(\cdot-ct-C_2)\Vert_{L^2} \Vert v_t\Vert_{L^2}\\
&\leq c\ \Vert v\Vert_{C([0,T]; L^2(\R))}\ \vert C_1-C_2\vert.
\end{align*}

\noindent After the substitution $\cdot-C_1 \to \cdot$ (analogously for $C_2$) the second and third part can similarly be estimated by 
\begin{align*}
\vert II+III \vert &= \vert \langle \hat{u}_x(\cdot -ct), \hat{u}(\cdot -ct+C_1)-\hat{u}(\cdot-ct)\rangle_{L^2}\\
&\quad- \langle \hat{u}_x(\cdot-ct), \hat{u}(\cdot-ct+C_2)-\hat{u}(\cdot -ct)\rangle_{L^2}\vert\\
&= \vert \langle \hat{u}_x(\cdot-ct), \hat{u}(\cdot-ct+C_1)-\hat{u}(\cdot-ct+C_2)\rangle_{L^2}\vert\\
&\leq \Vert \hat{u}_x(\cdot-ct)\Vert_{L^2} \Vert \hat{u}(\cdot-ct+C_1)-\hat{u}(\cdot-ct+C_2)\Vert_{L^2}\\
&\leq \Vert \hat{u}_x \Vert_{L^2}^2 \ \vert C_1-C_2\vert.
\end{align*}
Overall, this provides us with the $P$-a.s. existence and uniqueness of a classical global solution
$C \in C^1([0,T];\R)$.
\end{proof}

\noindent The representation 
\begin{align*}
u_t=\hat{u}(\cdot -ct-C(t))+\tilde{v}_t,
\end{align*}
i.e.
\begin{align}\label{eq:tildev1}
\tilde{v}_t= u_t-\hat{u}(\cdot-ct-C(t))=v_t+\left(\hat{u}(\cdot-ct)-\hat{u}(\cdot-ct-C(t))\right)
\end{align}

\noindent should now phenomenologically correspond to the representation established by Bressloff and Webber in \cite{Bressloff}.

\noindent Introducing the notation
\begin{align*}
\tilde{\hat{u}}(t) := \hat{u}(\cdot -ct-C(t)),\quad
\tilde{\hat{u}}_x(t) :=\hat{u}_x(\cdot -ct-C(t)) 
\end{align*}

\noindent it is obvious that $(\tilde{v}_t)_{t\in[0,T]}$ is again a process on $L^2(\R)$ and strong solution of the function-valued SDE

\begin{align}\notag
d\tilde{v}_t&= dv_t+\partial_t\hat{u}(\cdot-ct)\ dt- \partial_t\hat{u}(\cdot-ct-C(t))\ dt\\ \notag
&=[-v_t+w\ast \left(F(v_t+\hat{u}_t)-F(\hat{u}_t)\right)]\ dt -c \hat{u}_x(\cdot-ct) \ dt\\ \notag
&\quad + \left(c+\dot{C}(t)\right)\hat{u}_x(\cdot -ct-C(t)) \ dt +\sqrt{\eps} dW_t^Q \\ \notag
&=\big[-v_t+w\ast \left(F(v_t+\hat{u}_t)-F(\hat{u}_t)\right) 
  -\hat{u}_t+w\ast F(\hat{u}_t)+ \hat{u}(\cdot -ct-C(t)) \\ \notag
  &\quad - w \ast F(\hat{u}(\cdot -ct-C(t)))
  + \dot{C}(t)\hat{u}_x(\cdot -ct-C(t))\big] \ dt +\sqrt{\eps} dW_t^Q \\ \notag
  &=\left[-\tilde{v}_t + w\ast \left(F(v_t+\hat{u}_t)-F(\hat{u}_t)\right)+w\ast \left(F(\hat{u}_t)-F(\hat{u}(\cdot -ct-C(t)))\right)\right]dt\\ \notag
  &\quad 
  + \dot{C}(t)\ \hat{u}_x(\cdot -ct-C(t)) \ dt+ \sqrt{\eps} dW_t^Q \\ \notag
  &= \left[-\tilde{v}_t + w \ast \left(F(\tilde{v}_t+\tilde{\hat{u}}_t)-F(\tilde{\hat{u}}_t)\right)\right] dt 
  + \dot{C}(t)\ \hat{u}_x(\cdot -ct-C(t)) \ dt+ \sqrt{\eps} dW_t^Q \\
  &=: [ A(t) \tilde{v}_t + R_t(\tilde{\hat{u}}_t, \tilde{v}_t) ] dt + \sqrt{\eps} dW_t^Q \label{eq:tildev}
\end{align}

\noindent for the time- and $\omega$-dependent linear operator
\begin{align}\label{eq:defA}
A(t)v = -v + w\ast\left(F'(\tilde{\hat{u}}_t)\ v \right) 
       - m \ \langle \hat{u}_x(\cdot-ct-C(t)), v \rangle_{L^2} \cdot \hat{u}_x(\cdot-ct-C(t))
\end{align}
and the remainder $R_t$ as introduced in \eqref{remainder}. Here, the $\omega$-dependence of $A$ results from the $\omega$-dependence of the process $C$. Note that $\langle \hat{u}_x, v\rangle \cdot \hat{u}_x$ in the above formula denotes an orthogonal projection of $v$ onto the linear span generated by $\hat{u}_x(\cdot-ct-C(t))$ that, in the case $c\neq 0$, describes the infinitesimal tangential direction of the travelling wave $\hat{u}$ w.r.t. time, because
\[
\partial_t \hat{u} = -\hat{u} + w\ast F(\hat{u}) = -c\hat{u}_x.
\]  

To obtain an even deeper understanding of the behaviour of the stochastic travelling wave $u$, its stability and the exact influence of the noise strength $\sqrt{\eps}$ the following analysis works out the inner structure of the fluctuations $\tilde{v}$. To this end it is necessary to more thoroughly examine the family of linear operators $(A(t))_{t\in[0,T]}$.

\subsection{Properties of the associated family of linear operators}\label{sec:prop}
For $t\in[0,T]$ let $A^0(t)$ denote the linear operator on $L^2(\R)$ defined by 
\[
A^0(t) z = - z+ w\ast \left(F'(\tilde{\hat{u}}_t)\ z\right), \quad\; z\in L^2(\R).
\]
We work under the following 

\begin{ass}\label{ass:spectralgap}
There exist $\kappa_*>0$, $C_*>0$ such that
\begin{align*}\label{A1}
\langle A^0(t) z, z \rangle \leq -\kappa_* \Vert z \Vert^2+ C_*\left(\langle \tilde{\hat{u}}_x(t), z \rangle\right)^2 \quad \forall z\in L^2(\R).\tag{A1}
\end{align*}
Furthermore, in the sequel the relaxation rate $m$ is chosen such that 
\begin{align*}
 m>C_*. \tag{A2}\label{A2}
\end{align*}
\end{ass}
\bigskip

\noindent In the following Lemma relevant properties of the family of linear operators $(A(t))$ are proven. We denote by $\mathscr{L}(L^2(\R))$ the space of all bounded linear operators on $L^2(\R)$.
\begin{lem}\label{lem:A}
$A: [0,T] \times \Omega \rightarrow \mathscr{L}(L^2(\R))$ is well-defined, strongly measurable and strongly adapted.
Furthermore,
\begin{enumerate}[(i)]
\item For all $t\in[0,T]$, almost all $\omega \in \Omega$, $A(t)(\omega)$ is a bounded linear operator on $L^2(\R)$. 
\item For almost all $\omega \in \Omega$ the map $t\mapsto A(t)(\omega)$ is continuous in the uniform operator topology.
\end{enumerate}
In particular, $A$ $P$-a.s. generates an evolution family $\left(P(t,s)\right)_{0\leq s \leq t \leq T}$ on $L^2(\R)$ and for $m>C_*$:
\begin{itemize}
\item[(iii)] $\langle v, A(t)v \rangle_{L^2} \leq -\kappa_* \Vert v\Vert_{L^2}^2$.
\item[(iv)] $\Vert P(t,s) \Vert_{\mathscr{L}(L^2(\R))} \leq e^{-\kappa_*(t-s)}$ $\;\;\forall \ 0\leq s \leq t \leq T$.
\end{itemize}
\end{lem}

\begin{proof} Let $z\in L^2(\R)$. As a combination of measurable, adapted functions $A(\cdot,\cdot)z$ itself is measurable and adapted.

\noindent(i) Linearity of the operator is obvious. Furthermore, straightforward calculations yield the bound
\[
\sup_{r\in[0,T]} \Vert A(r) \Vert\leq \ 1+\Vert F'\Vert_{\infty}+ m \ \Vert \hat{u}_x \Vert^2_{L^2}.
\]

\noindent(ii) Let $s,t\in[0,T]$. We estimate
\begin{align*}
&\Vert A(t)z-A(s)z\Vert\\
 &\leq \big{\Vert} w\ast \left(F'(\tilde{\hat{u}}_t)\ z \right)- w\ast \left(F'(\tilde{\hat{u}}_s)\ z \right)\big{\Vert} + \big{\Vert} -m\ \langle \tilde{\hat{u}}_x(t), z \rangle \ \tilde{\hat{u}}_x(t) + m\ \langle \tilde{\hat{u}}_x(s), z \rangle \ \tilde{\hat{u}}_x(s)\big{\Vert}\\
&= \big{\Vert} w\ast\left(F'(\tilde{\hat{u}}_t)-F'(\tilde{\hat{u}}_s)\right) \ z \big{\Vert}\\
&\quad + m \
\big{\Vert} \langle \tilde{\hat{u}}_x(t), z \rangle \ \tilde{\hat{u}}_x(t)
- \langle \tilde{\hat{u}}_x(t), z \rangle \ \tilde{\hat{u}}_x(s)
+\langle \tilde{\hat{u}}_x(t), z \rangle \ \tilde{\hat{u}}_x(s)
-\langle \tilde{\hat{u}}_x(s), z \rangle \ \tilde{\hat{u}}_x(s) \big{\Vert}\\
&=: I + II
\end{align*}

\noindent Applying Jensen's inequality one obtains
\begin{align*}
I^2 &\leq \int_{\R}\int_{\R} w(x-y)\ \big{\vert} F'(\tilde{\hat{u}}_t(y))-F'(\tilde{\hat{u}}_s(y))\big{\vert}^2 \  z^2(y) \ dy \ dx \\
&\leq \Vert F''\Vert^2_{\infty} \ \int_{\R}\int_{\R} w(x-y)\ \big{\vert} \tilde{\hat{u}}_t(y)-
\tilde{\hat{u}}_s(y)\big{\vert}^2 \ z^2 (y) \ dy \ dx\\
&\leq \Vert F''\Vert^2_{\infty} \ \int_{\R}\int_{\R} w(x-y) \ \Vert \hat{u}\Vert_{C^1}^2 \ \vert ct-cs+C(t)-C(s)\vert^2\ z^2 (y) \ dy \ dx\\
&\leq \Vert F''\Vert^2_{\infty}\Vert \hat{u}\Vert_{C^1}^2 \ \left(c\ \vert t-s \vert+ \Vert C\Vert_{C^1} \vert t-s \vert\right)^2  \int_{\R} z^2(y) \int_{\R} w(x-y)\ dx \ dy \\
&\leq \tilde{c} \ \vert t-s \vert^2 \ \Vert z \Vert^2_{L^2}. 
\end{align*}
The second summand is estimated as follows:
\begin{align*}
II &\leq m \left(\Vert \langle \tilde{\hat{u}}_x(t), z\rangle \ \left(\tilde{\hat{u}}_x(t)-\tilde{\hat{u}}_x(s)\right) \Vert_{L^2}
+ \Vert \langle \tilde{\hat{u}}_x(t)-\tilde{\hat{u}}_x(s),z\rangle \ \tilde{\hat{u}}_x(s)\Vert_{L^2} \right)\\
&\leq m \ \left(\Vert \tilde{\hat{u}}_x(t)\Vert_{L^2} \ \Vert z \Vert_{L^2} \ \Vert\tilde{\hat{u}}_x(t)-\tilde{\hat{u}}_x(s)\Vert_{L^2}
+ \Vert \tilde{\hat{u}}_x(t)-\tilde{\hat{u}}_x(s)\Vert_{L^2} \ \Vert z \Vert_{L^2} \ \Vert \tilde{\hat{u}}_x(s)\Vert_{L^2} \right)
\end{align*}

\noindent According to Lemma \ref{lem:lipgrad} there exists a constant $\tilde{c}>0$ such that
\begin{align*}
\Vert \tilde{\hat{u}}_x(t)-\tilde{\hat{u}}_x(s)\Vert_{L^2} &\leq \tilde{c}\ \vert ct+C(t)-cs-C(s)\vert \leq \tilde{c}\ \left(c\ \vert t-s\vert + \Vert C \Vert_{C^1} \vert t-s \vert\right)\\
\end{align*}

\noindent Thus, $II$ is also of order $\vert t-s\vert$, which then yields
the continuity (even Lipschitz continuity) of $t\mapsto A(t)$ in the uniform operator norm. 
\smallskip

\noindent Given (i) and (ii) the existence of an evolution family $\left(P(t,s)\right)_{0\leq s \leq t\leq T}$ on $L^2(\R)$ generated by $\left(A(t)\right)_{t\in [0,T]}$ is now provided by \cite[Chapter 5, Theorem 5.1]{Pazy}.
\medskip

\noindent(iii)  The negative definiteness of the operator is a direct consequence of Assumption \ref{ass:spectralgap}.

\noindent(iv) Let $Q(t,s) = e^{\kappa_*(t-s)}  P(t,s),\;s\leq t$. Then,
\begin{align*}
\partial_t \Vert Q(t,s) z\Vert^2 &= 2 \ \langle Q(t,s)z, \ \kappa_* e^{\kappa_*(t-s)}P(t,s)z+ e^{\kappa_*(t-s)}A(t)P(t,s)z\rangle\\
&=2\ \langle Q(t,s) z , \ (A(t)+\kappa_*)\ Q(t,s)z\rangle\\
&\leq -2\ \kappa_* \Vert Q(t,s) z \Vert^2 + 2\ \kappa_* \Vert Q(t,s) z \Vert^2\\
&=0
\end{align*}
which implies
\[
\Vert Q(t,s)z\Vert^2\leq \Vert Q(s,s)z \Vert^2 = \Vert z\Vert^2.
\]
\end{proof}

\subsection{Ornstein-Uhlenbeck decomposition of the \mbox{$L^2$-component}}\label{sec:5}
To even better characterise the behaviour of the $L^2(\R)$-valued  fluctuations $\tilde{v}$ we will derive a decomposition into an Ornstein-Uhlenbeck process $(Z_t)_{t\in[0,T]}$ and
a corresponding remainder process $(y_t^{\eps})_{t\in[0,T]}$, which turn out to display different orders of the intrinsic dynamics of $\tilde{v}$. To this end let us first introduce 
\begin{align}\label{eq:veps}
\tilde{v}^{\eps}_t := \frac{1}{\sqrt{\eps}} \ \tilde{v}_t,
\end{align}

\noindent which by equation \eqref{eq:tildev} satisfies the SDE

\begin{align*}
\begin{cases}
d\vet
&= \left[A(t)\ \tilde{v}^{\eps}_t +\sqrt{\eps}\ R_t^{\eps}(\tilde{\hat{u}}_t, \vet)\right] \ dt + dW_t^Q\\
\tilde{v}^{\eps}(0)&= \frac{1}{\sqrt{\eps}}\ v_0
\end{cases}
\end{align*}
with
\begin{align*}
R_t^{\eps}(\tilde{\hat{u}}_t, \vet)= w\ast \left(\frac 12 F''(\xi(\tilde{\hat{u}}_t, \sqrt{\eps} \vet))(\vet)^2\right).
\end{align*}
As before, $\xi(u,v)$ denotes an intermediate point between $u$ and $v$. 

To work out the inner structure of the solution w.r.t. different order terms we introduce the decomposition
\begin{align*}
\vet = Z_t + \yet,
\end{align*}
where $(Z_t)$ satisfies the linear SDE

\begin{align}\label{eq:Z}
\begin{cases}
dZ_t&= A(t) Z_t \ dt+ dW_t^Q\\
Z_0&=0
\end{cases}
\end{align}
with $\omega$-dependent operator $A(t)$ and $(\yet)$ consequently solves the pathwise function-valued ODE
\begin{align}\label{eq:yt}
\begin{cases}
y'_{\eps}(t) = A(t)y_{\eps}(t) + \sqrt{\eps}\ R_t^{\eps}(\tilde{\hat{u}}_t, Z_t+y_{\eps}(t))\\
y_{\eps}(0)= \frac{1}{\sqrt{\eps}} v_0
\end{cases}
\end{align}

\noindent The well-posedness of this decomposition is provided by the following existence and uniqueness theorems:

\begin{thm}\label{thm:strongz}
There exists a unique mild solution $Z$ of equation \eqref{eq:Z}
satisfying
\begin{align}\label{eq:strongz}
Z_t= \int_0^t A(s)Z_s \ ds + W_t^Q.
\end{align}
Thus, $Z$ is also a strong solution.
\end{thm}

\begin{proof}
To prove existence and uniqueness of a mild solution to equation \eqref{eq:Z}
set
\[
B(t,\omega,Z):= A(t)(\omega)Z
\]
and verify that $B$ satisfies a Lipschitz- as well as linear growth condition as required in the standard existence and uniqueness result for mild solutions (\cite[Theorem 7.4]{DaPrato}). Since one can formally enhance equation \eqref{eq:Z} by the linear operator $A=0$, generating the semigroup
$T_t= \id, \; t\in[0,T]$, the unique mild solution ${Z\in L^{\infty}([0,T];L^2(\R))}$ even satisfies the identity
\[
Z_t=\int_0^t A(s)Z_s \ ds+W_t^Q.
\]
\end{proof}

\noindent Given the unique strong solutions $(v_t^{\eps})_{t\in[0,T]}$ and $(Z_t)_{t\in[0,T]}$ the following theorem is a direct consequence: 

\begin{thm}
For all $T>0$ there exists a differentiable $L^2(\R)$-valued process $(y^{\eps}_t)$,
which is the unique strong solution of the pathwise ODE \eqref{eq:yt}.
\end{thm}

Even though Theorem \eqref{thm:strongz} ensures existence and uniqueness of a strong solution $(Z_t)$, equation \eqref{eq:strongz} only yields an implicit representation of
the process. As already shown in Lemma \ref{lem:A}, the family $(A(t))$ $P$-almost surely generates an evolution family $(P(t,s))$, which will even allow us to find an explicit mild-solution-like representation of $(Z_t)$ via the following representation formula for weak solutions. This formula has been introduced as so-called ``pathwise mild solution'' in a much more general setting (cf. \cite{Pronk}) and manifests a way to pass around the difficulty of defining a mild-solution-like stochastic convolution in the case where the integrand cannot be assumed to be adapted. This indeed occurs if the operator $A(t)$ depends on the underlying probability space.

\begin{thm}\label{thm:mildsol2}
The process
$Z:[0,T]\times\Omega\to L^2(\R)$ defined by
\begin{align*}
Z_t&=P(t,0)Z_0 + \int_0^t P(t,r)A(r)W^Q_r \ dr + W_t^Q\\ 
    &= P(t,0)Z_0 + \int_0^t P(t,r)A(r)(W^Q_r-W^Q_t) \ dr + P(t,0)W^Q_t
\end{align*}
is an adapted weak solution of the linear SDE
\begin{align*}
dZ_t&= A(t) Z_t \ dt+ dW_t^Q
\end{align*}
with $Z_0\in L^2(\R)$.

\end{thm}

\noindent Using this representation, the Ornstein-Uhlenbeck process can now be controlled by the following pathwise order estimate, which yields a significant $\eps$-independent bound on $(Z_t)$. This is not trivial since the operator $A(t)$ indirectly depends on $\eps$: 

\begin{lem}\label{prop:z}
Let $\eta \in(0,\frac 12)$.
The unique mild (and also strong) solution of equation \eqref{eq:Z}
satisfies
\[
\sup_{t\in[0,T]} \Vert Z_t\Vert_{L^2(\R)} \leq C \ \xi,
\]
with a constant
$C=C(\kappa_*, \eta, \Vert F'\Vert_{\infty}, \Vert \hat{u}_x\Vert_{L^2})$ and $\xi = \Vert W^Q \Vert_{C^{\eta}([0,T]; L^2(\R))}$. In particular, we have $\xi<\infty$ almost surely.
\end{lem}

\begin{proof}
By Theorem \ref{thm:mildsol2} the unique strong solution can be represented as
\[
Z_t=\int_0^t P(t,r)A(r)(W^Q_r-W^Q_t) \ dr + P(t,0) W^Q_t
\]

\noindent and for each $\eta \in (0, \frac 12)$ we have $W^Q \in C^{\eta}([0,T]; L^2(\R))$ a.s. (refer to \cite{DaPrato}).

\noindent Consequently, for $t\in[0,T]$ and $\xi = \Vert W^Q \Vert_{C^{\eta}([0,T]; L^2(\R))}$:
\begin{align*}
\Vert Z_t\Vert_{L^2} &\leq \int_0^t \Vert P(t,r)\Vert \ \Vert A(r)(W^Q_r-W^Q_t)\Vert \ dr + \Vert P(t,0) W^Q_t\Vert \\
&\leq \sup_{r\in[0,T]} \Vert A(r)\Vert \int_0^t e^{-\kappa_*(t-r)} \ \Vert W^Q_r-W^Q_t\Vert \ dr + e^{-\kappa_* t}\ \Vert W^Q_t\Vert\\
&\leq \sup_{r\in[0,T]} \Vert A(r) \Vert \int_0^t e^{-\kappa_*(t-r)} \ \Vert W^Q\Vert_{C^{\eta}} \ \vert t-r\vert^{\eta} \ dr + e^{-\kappa_* t}\ \Vert W^Q\Vert_{C^{\eta}} \ t^{\eta}\\
&\leq \sup_{r\in[0,T]} \Vert A(r) \Vert \left[\int_0^t e^{-\kappa_*(t-r)}\ \vert t-r\vert^{\eta} \ dr + e^{-\kappa_* t} \ t^{\eta}\right] \ \xi\\
&=: \sup_{r\in[0,T]} \Vert A(r) \Vert \left[ I+II\right]\ \xi.
\end{align*}

\noindent Lemma \ref{lem:A}(i) yields the bound
\[
\sup_{r\in[0,T]} \Vert A(r) \Vert\leq \ 1+\Vert F'\Vert_{\infty}+ m \ \Vert \hat{u}_x \Vert^2_{L^2},
\]
which, in particular, does not depend on $\omega$. Furthermore, after suitable substitutions
\begin{align*}
I&=\int_0^t e^{-\kappa_*y} y^{\eta} \ dy \leq \kappa_*^{-\eta-1} \ \Gamma(\eta+1),
\end{align*}
where $\Gamma$ denotes the Gamma function
\[
\Gamma(x)=\int_0^{\infty} t^{x-1} e^{-t} \ dt \;\;\text{ for }x\in \R_+.
\]

\noindent Likewise, term II is a bounded function on $[0,T]$, thus
\begin{align*}
II\leq \sup_{[0,T]}  e^{-\kappa_* t} \ t^{\eta} <\infty.
\end{align*}

\noindent Combining all results, there exists a positive constant
 ${C=C(\kappa_*, \eta, \Vert F'\Vert_{\infty}, \Vert \hat{u}_x\Vert_{L^2})}$ such that 
\[
\sup_{t\in[0,T]} \Vert Z_t\Vert_{L^2(\R)} \leq C \ \xi.
\]
\end{proof}


\section{Stability results}\label{sec:stab}
Summing up, we derived a decomposition
\[
\tilde{v}_t=\sqrt{\eps}\ Z_t+ \sqrt{\eps}\ \yet
\]
such that $(Z_t)$ is $P$-a.s. of order $\mathcal{O}(1)$. Characterising the behaviour of the stochastic travelling wave $u$ for small noise strength $\sqrt{\eps}$ our main theorem of this section will show that the remainder process $(\yet)$ is of lower order than $\mathcal{O}(1)$, i.e.
$\lim_{\eps\to 0} \yet = 0$ in $L^2(\R)$ uniformly w.r.t. $t\in[0,T]$ $P$-a.s. The relevance of this result is that with high probability we have the decomposition

\[
u_t=\tilde{\hat{u}}_t + \sqrt{\eps}\ Z_t + \text{ lower order terms}.
\]  
which, in comparison to the classical stability analysis via Evans functions (as conducted in \cite{Coombes} for a Heaviside nonlinearity), yields an alternative approach to the stability of fluctuating travelling waves.

\noindent From now on let us assume $v_0=0$, hence also $y_{\eps}(0)=0$. 

\begin{lem}\label{lem:ord}
There exists a constant $c>0$ such that

\[
\forall v \in L^2(\R), \ t\in[0,T]: \quad
\Vert R^{\eps}_t(\tilde{\hat{u}}_t,v)\Vert_{L^2} \leq c \ \Vert v \Vert_{L^2}^2.
\]

\noindent Explicitly, the constant is given by 
\[
c= \frac 12\ \Vert F''\Vert_{\infty} \left(\sup_{y,\tilde{y}\in\R} \int_{\R} w(x-y)w(x-\tilde{y}) \ dx\right)^{1/2}\leq \frac 12\ \Vert F''\Vert_{\infty}
\Vert  w \Vert_{\infty}^{1/2}.
\]
\end{lem}

\begin{proof}
Straightforward calculations. 
\end{proof}

\noindent For the proof of our main stability result (Theorem \ref{thm:stab}) the following bound on $(\yet)$ is already a crucial achievement:

\begin{thm}\label{prop:y}
Let $\sqrt{\eps}<\frac{\kappa_*}{4c}$, where $c$ is the constant from Lemma \ref{lem:ord} and define 
\[
Z:= \sup_{t\in [0,T]}\Vert Z_t\Vert_{L^2}^2.
\]
On the set
$\Omega_{\eps}=\big{\{}\omega \in \Omega\ \vert \ Z<\frac{\kappa_*}{8c\sqrt{\eps}}\big{\}}$ we obtain the following uniform bound on $y_t^{\eps}$:

\[
\sup_{t\in[0,T]} \Vert y_t^{\eps} \Vert_{L^2}^2 \leq \frac 32 \ Z.
\]

\noindent In the limit $\eps\downarrow 0$ this bound even holds for $P$-almost all paths $\omega \in\Omega$, more precisely:
\[
\lim_{\eps \to 0} P[\Omega_{\eps}]=1.
\] 
\end{thm}

\begin{proof}
Applying Lemma \ref{lem:A} and Lemma \ref{lem:ord}, the process $\Vert y_t^{\eps} \Vert^2$ satisfies the following differential inequality: 

\begin{align}\label{diff1}\notag
\frac 12 \ \partial_t \Vert y_t^{\eps} \Vert_{L^2}^2 &= \langle y_t^{\eps}, A(t)y_t^{\eps}\rangle_{L^2} + \sqrt{\eps}\ \langle y_t^{\eps}, R_t^{\eps}(\tilde{\hat{u}}_t, Z_t+y_t^{\eps})\ \rangle_{L^2}\\\notag
&\leq -\kappa_* \Vert y_t^{\eps}\Vert_{L^2}^2 +\sqrt{\eps}\ \Vert y_t^{\eps}\Vert_{L^2}\ c\ \Vert Z_t+y_t^{\eps}\Vert_{L^2}^2\\
&\leq -\kappa_* \Vert y_t^{\eps}\Vert_{L^2}^2 + 2c \sqrt{\eps}\ \Vert y_t^{\eps}\Vert_{L^2}\left(\Vert Z_t\Vert_{L^2}^2 +\Vert y_t^{\eps}\Vert_{L^2}^2\right)
\end{align}

\noindent W.l.o.g. assume $c=1$. For $\sqrt{\eps}<\frac{\kappa_*}{4}$ the right-hand side of \eqref{diff1} can be further estimated from above by 

\begin{align*}
\frac 12 \ \partial_t \Vert y_t^{\eps} \Vert_{L^2}^2 
&\leq -\kappa_* \Vert y_t^{\eps}\Vert_{L^2}^2+ \sqrt{\eps}\left(\Vert y_t^{\eps}\Vert_{L^2}^2 +\Vert Z_t \Vert_{L^2}^4\right)
+\sqrt{\eps}\left(\Vert y_t^{\eps}\Vert_{L^2}^2 +\Vert y_t^{\eps} \Vert_{L^2}^4\right)\\  
&\leq -\frac{\kappa_*}{2} \Vert y_t^{\eps} \Vert_{L^2}^2+\sqrt{\eps} \Vert Z_t \Vert_{L^2}^4+ \sqrt{\eps} \Vert y_t^{\eps} \Vert_{L^2}^4 
\end{align*}

\noindent Writing
$ g_{\eps}(t):= \Vert y_t^{\eps} \Vert_{L^2}^2$ one obtains the following differential inequality

\begin{align}\label{eq:geps}
\dot{g}_{\eps}(t) \leq -\kappa_* g_{\eps}(t) + 2\sqrt{\eps} Z^2 + 2\sqrt{\eps} g_{\eps}^2(t)
\end{align}

\noindent By the comparison principle for ODE this problem is now solved in the case of true equality, i.e.
\[
\frac{\dot{g}_{\eps}(t)}{-\kappa_*g_{\eps}(t)+2\sqrt{\eps}g_{\eps}^2(t)+2\sqrt{\eps}Z^2} =1
\]
Integrating over $[0,t]$ and carrying out a suitable substitution yields
\begin{align}
t
&= \int_0^{g_{\eps}(t)} \frac{1}{-\kappa_* \ g + 2\sqrt{\eps} \ g^2 + 2\sqrt{\eps}Z^2} \ dg \label{sf}
\end{align}

\noindent Note that
$\Delta:= 16\eps  Z^2 - \kappa_*^2<0$ on the set 
$\Omega_{\eps}$. 
Thus, on this particular set of paths the integral \eqref{sf} is given by

\begin{align}\label{eq1}
t=\frac{1}{\sqrt{-\Delta}} \log\left(\frac{4\sqrt{\eps}g_{\eps}(t)-\kappa_*-\sqrt{-\Delta}}{4\sqrt{\eps}g_{\eps}(t)-\kappa_*+\sqrt{-\Delta}}\right)
-\frac{1}{\sqrt{-\Delta}}\log\left(\frac{-\kappa_*-\sqrt{-\Delta}}{-\kappa_*+\sqrt{-\Delta}}\right)
\end{align}

\noindent since $g_{\eps}(0)=0$. Let 
\[
M_{\eps}:= \frac{\kappa_*+\sqrt{-\Delta}}{\kappa_*-\sqrt{-\Delta}}.
\] 
It is important to remark that $M_{\eps} \in (1,\infty)$ P-a.s. since
$Z>0$ a.s. and hence \mbox{$-\Delta=\kappa_*^2-16\eps Z^2<\kappa_*^2$} a.s.
Now, equation \eqref{eq1} is equivalent to

\begin{align*}
&\quad g_{\eps}(t) = \frac{-(\kappa_*+\sqrt{-\Delta})+(\kappa_*-\sqrt{-\Delta})e^{\sqrt{-\Delta}\ t} M_{\eps}}{4\sqrt{\eps}\left(e^{\sqrt{-\Delta}\ t} M_{\eps}-1\right)}
\end{align*}
Note that an explosion of $g_{\eps}(t)$ is excluded by the fact that $M_{\eps}>1$  a.s. Given these considerations and dropping the first negative summand in the numerator of $g_{\eps}(t)$ we further estimate
\begin{align*}
g_{\eps}(t)&\leq \frac{(\kappa_*-\sqrt{-\Delta})\ e^{\sqrt{-\Delta}\ t} M_{\eps}}{4\sqrt{\eps}\left(e^{\sqrt{-\Delta}\ t} M_{\eps}-1\right)}=:f_{\eps}(t)
\end{align*}

\noindent It is easy to see that $f_{\eps}(t)$ is a non-increasing function, hence attains its supremum in $t=0$. This implies
\begin{align*}
\sup_{t\in[0,T]} g_{\eps}(t) \leq f_{\eps}(0)= \frac{(\kappa_*-\sqrt{-\Delta})}{4\sqrt{\eps}} \cdot \frac{M_{\eps}}{M_{\eps}-1} = I \cdot II, \text{ say.}
\end{align*}

\noindent On $\Omega_{\eps}$ one is able to estimate
\[
\sqrt{\kappa_*^2-16\eps Z^2}\geq \sqrt{\kappa_*^2}-\sqrt{16\eps Z^2} = \kappa_*-4\sqrt{\eps}Z
\]
which allows us to bound $I$ as follows:

\begin{align*}
I=\frac{\kappa_*-\sqrt{\kappa_*^2-16\eps Z^2}}{4\sqrt{\eps}}\leq \frac{\kappa_*-(\kappa_*-4\sqrt{\eps}Z)}{4\sqrt{\eps}}=Z.
\end{align*}

\noindent Turning our attention to $II$ and again restricting ourselves to the paths in $\Omega_{\eps}$ we obtain

\begin{align*}
II= \frac{\kappa_*+\sqrt{-\Delta}}{2\sqrt{-\Delta}}= \frac{\kappa_*}{2\sqrt{\kappa_*^2-16\eps Z^2}} + \frac 12
\leq \frac{\kappa_*}{2\kappa_*-8\sqrt{\eps}Z}+ \frac 12 \leq \frac 32.
\end{align*}

\noindent To prove that in the limit $\eps \to 0$ the obtained pathwise bound even holds for almost every $\omega\in\Omega$
note that for decreasing $\eps$ the sets $\Omega_{\eps}$ constitute an ascending sequence of sets converging to the event $\{Z<\infty\}$. Therefore,
\[
P[\Omega_{\eps}]\underset{\eps\to 0}{\longrightarrow} P[Z<\infty]=1,
\]
since $Z$ is an integrable random variable. 
\end{proof}

This auxiliary result now allows us to state the main theorem of this section. Similar results on the stability of certain macroscopic dynamics modelled by stochastic partial differential equations on bounded domain have also been obtained by Bl\"omker in \cite{Bloem1}, \cite{Bloem2}, \cite{Bloemker}. 

\begin{thm}\label{thm:stab}
Let $q\in(0, \frac 12)$ and let $c$ denote the constant from Lemma \ref{lem:ord}. Define the stopping time $\tau$ by
\[
\tau= \inf \Big{\{}t\geq 0 \big{\vert} \int_0^t \Vert \tilde{v}_s^\varepsilon \Vert_{L^2}^2 \ ds > \varepsilon^{-q}\Big{\}}\wedge T
\]
Then we obtain the order estimate
\[
\sup_{t\in[0,\tau]} \Vert y_t^\varepsilon \Vert _{L^2} \leq c \ \varepsilon^{1/2-q}.
\]

\noindent Moreover,
\[
\lim_{\varepsilon \to 0} P[\tau=T] = 1.
\]
\end{thm}

\begin{proof}
Let $t<\tau$. The process $y^\varepsilon$ satisfies
\[
y_t^\varepsilon =\sqrt{\varepsilon} \int_0^t P(t,s) R_s^\varepsilon(\tilde{\hat{u}}_s, \tilde{v}^\varepsilon_s)\ ds.
\]
With Lemma \ref{lem:A} and \ref{lem:ord} we are able to estimate  

\begin{align*}
\Vert y_t^{\varepsilon} \Vert_{L^2} 
&\leq \sqrt{\eps}\int_0^t e^{-\kappa_*(t-s)}\ c \Vert \tilde{v}_s^\eps \Vert _{L^2}^2 \ ds \leq c\ \eps^{1/2-q} 
\end{align*}

\noindent Hence,
\[
\sup_{t \in [0,\tau]}\Vert y_t^\eps\Vert_{L^2} \leq c\ \eps^{1/2-q}.
\] 
To prove the convergence $P[\tau=T]\to 1 \text{ for }\eps \to 0$ define the set
\[
\Omega^*:= \bigg{\{}\omega \in \Omega \ \bigg{\vert} \ \sup_{t\in[0,T]}\Vert Z_t\Vert_{L^2}^2\leq \frac{\eps^{-q}}{4T},\;\;\sup_{t\in[0,T]}\Vert y_t^\eps\Vert_{L^2}^2\leq \frac{\eps^{-q}}{4T}\bigg{\}}
\]
which can be shown to be a subset of $\{\tau=T\}$. Indeed,
\[
\{\tau=T\}
=\Big{\{}\omega\in\Omega \Big{|}\ \int_0^T \Vert \tilde{v}_s^\eps\Vert_{L^2}^2\ ds \leq \eps^{-q}\Big{\}} 
\]

\noindent and on $\Omega^*$ one is able to estimate

\begin{align*}
\int_0^T \Vert \tilde{v}_s^\eps\Vert_{L^2}^2\ ds 
&\leq 2T \left(\sup_{t\in[0,T]}\Vert Z_t\Vert_{L^2}^2+ \sup_{t\in[0,T]}\Vert y_t^\eps\Vert_{L^2}^2\right)
\leq \eps^{-q}. 
\end{align*}

\noindent In the following step we prove that in the limit $\eps\downarrow 0$ even the smaller set $\Omega^*$ has full measure: Let 
$P_{\Omega_{\eps}}=P[\ \cdot \ \vert \Omega_{\eps}]$
denote the conditional probability distribution on the set $\Omega_\eps$ with corresponding conditional expectation
$E_{\Omega_{\eps}}=E[\ \cdot \ \vert \Omega_{\eps}]$.

\noindent By Lemma \ref{prop:z} the random variable $Z=\sup_{t\in[0,T]}\Vert Z_t\Vert^2_{L^2}$ is bounded by the integrable majorant $C^2\xi^2$, which, in particular, is independent of $\eps$. 
With Markov's inequality and Proposition \ref{prop:y} one obtains
\begin{align*}
P_{\Omega_{\eps}}[\Omega^*]&\geq 1-P_{\Omega_{\eps}}\bigg[\sup_{t\in [0,T]}\Vert Z_t\Vert_{L^2}^2 >\frac{1}{4\ T\eps^q}\bigg]
-P_{\Omega_{\eps}}\bigg[\sup_{t\in [0,T]}\Vert y_t^{\eps}\Vert_{L^2}^2 >\frac{1}{4\ T\eps^q}\bigg]\\[11pt]
&\geq 1-4\ T\eps^q \  E_{\Omega_{\eps}}\bigg[\sup_{t\in [0,T]}\Vert Z_t\Vert_{L^2}^2\bigg]
-4\ T\eps^q \ E_{\Omega_{\eps}}\bigg[\sup_{t\in [0,T]}\Vert y_t^{\eps}\Vert_{L^2}^2\bigg]\\[11pt]
&\geq 1-4\ T\eps^q \ \frac{E[C^2\xi^2 \ ; \Omega_{\eps}]}{P[\Omega_{\eps}]}
-4\ T\eps^q \ \frac{E[\sup_{t\in [0,T]}\Vert y_t^{\eps}\Vert_{L^2}^2 \ ; \Omega_{\eps}]}{P[\Omega_{\eps}]}\\[11pt]
&\geq 1-\frac{4\ T\eps^q}{P[\Omega_{\eps}]}\ C^2 \ E[\xi^2]
-\frac{4\ T\eps^q}{P[\Omega_{\eps}]}\ \frac 32 \ E[Z]\\[11pt]
&{\longrightarrow}\; 1\quad \text{  as } \eps\to 0.
\end{align*} 

\noindent In conclusion, the limit behaviour of the original probability measure $P$ is immediately determined by
\[
P[\Omega^*]\geq P_{\Omega_{\eps}}[\Omega^*]\ P[\Omega_{\eps}]\underset{\eps \to 0}{\longrightarrow} 1,
\]
which suffices to prove the assertion.
\end{proof}

\bibliographystyle{siam}
\bibliography{biblio}
\end{document}